\documentclass[11pt]{amsart}
  \usepackage{amsmath, amscd, amsthm, amssymb}
	\usepackage{xypic}
	\usepackage{cite, url}

\newcommand{\SL}{{\mathrm{SL}}}
\newcommand{\PGO}{{\mathrm{PGO}}}
\newcommand{\SO}{{\mathrm{SO}}}
\newcommand{\Spin}{{\mathrm{Spin}}}
\newcommand{\Hsp}{{\mathrm{HSpin}}}
\newcommand{\PSp}{{\mathrm{PSp}}}
\newcommand{\Sp}{{\mathrm{Sp}}}
\newcommand{\Gm}{{\mathbb{G}_m}}
\newcommand{\Hom}{{\mathrm{Hom}}}
\newcommand{\Invd}{{\mathrm{Inv}^3(G,\mathbb{Q}/\mathbb{Z}(2))_{\mathrm{dec}}}}
\newcommand{\Invi}{{\mathrm{Inv}^3(G,\mathbb{Q}/\mathbb{Z}(2))_{\mathrm{ind}}}}
\newcommand{\Invn}{{\mathrm{Inv}^3(G,\mathbb{Q}/\mathbb{Z}(2))_{\mathrm{norm}}}}
\newcommand{\Invdp}{{\mathrm{Inv}^3(G',\mathbb{Q}/\mathbb{Z}(2))_{\mathrm{dec}}}}
\newcommand{\Invip}{{\mathrm{Inv}^3(G',\mathbb{Q}/\mathbb{Z}(2))_{\mathrm{ind}}}}
\newcommand{\Invnp}{{\mathrm{Inv}^3(G',\mathbb{Q}/\mathbb{Z}(2))_{\mathrm{norm}}}}
\newcommand{\Z}{\mathbb{Z}}
\newcommand{\Q}{\mathbb{Q}}
\DeclareMathOperator{\ord}{ord}
\DeclareMathOperator{\Dec}{Dec}
\DeclareMathOperator{\res}{res}
\DeclareMathOperator{\CH}{CH}

\DeclareMathOperator{\ch}{char}

\newtheorem{thm}{Theorem}[section]
\newtheorem*{thm*}{Theorem}
\newtheorem{prop}[thm]{Proposition}
\newtheorem{cor}[thm]{Corollary}
\newtheorem{lem}[thm]{Lemma}

\theoremstyle{definition}

\theoremstyle{remark}
\newtheorem{rmk}[thm]{Remark}
\newtheorem*{rmk*}{Remark}

\title[Degree 3 Cohomological Invariants of Semisimple Groups]{Degree 3 Cohomological Invariants of Groups That Are Neither Simply Connected nor Adjoint}
\author{Hernando Bermudez} \email{hbermud@emory.edu}\address{Department of Mathematics \& Computer Science, Emory University, Atlanta, GA
30322, USA}
\author{Anthony Ruozzi} \email{anthony@mathcs.emory.edu} \address{Department of Mathematics \& Computer Science, Emory University, Atlanta, GA
30322, USA}
\date{}

\begin{document}
\begin{abstract}
In a recent paper A. Merkurjev constructed an exact sequence which includes as one of the terms the group of degree 3 normalized cohomological invariants of a semisimple algebraic group $G$, greatly extending results of M. Rost for simply connected quasi-simple groups. Furthermore, in the aforementioned paper, Merkurjev uses his exact sequence to determine the groups of invariants for all semisimple adjoint groups of inner type. The goal of this paper is to use Merkurjev's sequence to compute the group of invariants for the remaining split cases, namely groups of types $A$ and $D$ that are neither simply connected nor adjoint. This description not only demonstrates the existence of many previously unstudied invariants but also allows us to extend several known results which relate these invariants to the Rost invariant and algebras with involution.
\end{abstract}
\maketitle
\section{Introduction}

Let $G$ be a linear algebraic group over a field $F$. The group of degree $n$ cohomological invariants of $G$ with values in a $\mathrm{Gal}(F_{\mathrm{sep}}/F)$-module $A$ is the set of natural transformations of functors
\[I:H^1(-,G)\to H^n(-,A).\]
An invariant $I$ is called \emph{normalized} if $I(e)=0$ where $e$ is the trivial $G$-torsor. The object of interest for us is the group $\Invn$ of normalized invariants of degree 3 with values in the group $\Q/\Z(2)$ which is defined as the direct sum of the colimit over $n$ of the Galois modules $\mu_n^{\otimes 2}$ and a $p$-component defined via logarithmic de Rham-Witt differentials in the case $p=\ch(F)>0$ (see \cite[I.5.7]{Ill}). These invariants were determined by M. Rost in the case that $G$ is quasi-simple simply connected, and recently by A. Merkurjev \cite{ssinv2} in the case where $G$ is adjoint of inner type. In fact Merkurjev does quite a bit more, namely, he provides an exact sequence involving the degree $3$ invariants of a semisimple group.

In this paper we use Merkurjev's exact sequence to study the remaining cases of split quasi-simple groups, namely the groups $G=\Hsp_{4n}$\footnote{Recall that $\mathrm{Spin}_{4n}$ has three central subgroups of order $2$. The quotient by one of them is $\mathrm{SO}_{4n}$. In our notation, $\Hsp_{4n}$ is the quotient by either of the other two.} and $G=\SL_n/\mu_m$ (note that the case $G=\SO_{2n}$ has also been computed, cf. \cite[Part 1, Ch.~VI]{GMS}). Our study of the invariants for $\SL_n/\mu_m$ results in many new degree three invariants that have never been discussed in the literature. We describe these, when possible, by restricting the invariant to a suitable subgroup. For $\Hsp_{16}$, much more is known. The description of the invariants for $\Hsp_{4n}$ allows us to recover these results as well as extend them to arbitrary $n$. Of particular interest are a formula for an \textquotedblleft indecomposable\textquotedblright  invariant of $\Hsp_{4n}$ in terms the invariant for $\PSp_{2n}$ and the Rost invariant of a twisted $\Spin$ group and an extension of the Arason invariant $e_3$ to algebras with orthogonal involution.

\section{Decomposable Invariants}
Let $G$ be a semisimple group over a field $F$. Then there is an exact sequence \cite[Thm.~3.9]{ssinv2},
\begin{align}
0\to \CH^2(BG&)_{\mathrm{tors}}\to H^1(F,\hat{C}(1))\xrightarrow{\text{$\sigma$}} \label{thm:mes}\\
&\Invn\to Q(G)/\Dec(G)\xrightarrow{\text{$\theta_G^*$}} H^2(F,\hat{C}(1)). \nonumber
\end{align}

As Merkurjev observed, this exact sequence describes two types of invariants of very different natures. The first are the {\it decomposable invariants} defined as
\[\Invd := \mathrm{Im}(\sigma)\]
where $\sigma$ is the map from the sequence (\ref{thm:mes}). These invariants can be easily understood for our cases of interest by following Merkurjev's arguments for adjoint groups; namely,  we define a map $\alpha_G$ as follows. Let $\tilde{G}$ be the universal cover of $G$ with kernel $C$. For any character $\chi \in \hat{C}(F)$ we consider the pushout
\[\begin{CD}
1 @>>> C @>>> \tilde{G} @>>> G @>>> 1\\
@. @V\chi VV @VVV @| @.\\
1 @>>> \Gm @>>> H_{\chi} @>>> G @>>> 1\\
\end{CD}\]
Define a morphism
\[\alpha_G:\mathrm{H}^1(F,G) \rightarrow \Hom(\hat{C}(F),\mathrm{H}^2(F, \Gm))\]
by sending a torsor $E$ to the map $\chi \mapsto \partial(E)$ where $\partial: \mathrm{H}^1(F,G) \rightarrow \mathrm{H}^2(F, \Gm)$ is the connecting homorphism for the bottom sequence in the diagram. A homomorphism $a \in \Hom(\hat{C}(F),\mathrm{H}^2(F, \Gm))$ will be called {\it admissible} if $\mathrm{ind}(a(\chi)) \mid \mathrm{ord}(\chi)$ for every $\chi\in \hat{C}(F)$.

\begin{prop}
For the two families of groups $G=\SL_n/\mu_m$ and $G = \Hsp_{4n}$,
\[\Invd \simeq \hat{C} \otimes F.\]
\end{prop}
\begin{proof}
First, we show that every admissible map is in the image of $\alpha_G$. Following Merkurjev's arguments, it suffices to find a simply connected subgroup, $G' \subset \tilde{G}$ of type $\mathrm{A}_n$ that contains $C$. This can be done by inspection or exactly as in \cite[Prop.~2.4/2.6]{ssinv2}. The statement then follows as in the proof of \cite[Thm.~4.2]{ssinv2}.
\end{proof}

\section{Indecomposable Invariants}
The second, and more interesting, class of invariants are the {\it indecomposable invariants} defined as
\[\Invi := \Invn/\Invd.\]
We note that the elements of this group are not cohomological invariants in the sense described in the introduction. However, it is possible to define an invariant in the sense of \cite{GMS} from an indecomposable invariant by considering them as maps $H^1(F,G)\to H^3(F,\Q/\Z(2))/P$, where $P$ is the subgroup generated by cup products of Tits algebras of $G$ and elements of the field as in \cite[p.~11]{ssinv2}. The sequence
\[\begin{CD}
1 @>>> \Invd @>>> \Invn @>>> \Invi @>>> 1
\end{CD}\]
is actually functorial in $G$ as noted in the comments after \cite[Rem. 3.10]{ssinv2}. That is, for another group $G'$ and a map $G' \rightarrow G$, the diagram
\begin{equation}\label{eqn:decinv}\xymatrix{
\Invd\ar[d]\ar[r]&\Invn\ar[d]\ar[r]&\Invi\ar[d]\\
\Invdp\ar[r]&\Invnp\ar[r]&\Invip\\}
\end{equation}
commutes.

To understand the indecomposable invariants, we need to describe the groups $Q(G)$ and $\Dec(G)$. Let $T$ be a maximal torus in $G_{\mathrm{sep}}$, and let $\Lambda$ be the $\Gamma=\mathrm{Gal}(F_{\mathrm{sep}}/F)$-lattice corresponding to $T$, under the usual equivalence of categories between tori and their character modules, equipped with the $*$-action. This action permutes a system of simple roots, cf. \cite[\S 27.A]{KMRT}. It follows that $\Lambda_r \subset \Lambda \subset \Lambda_w$ where $\Lambda_r$ and $\Lambda_w$ are, respectively, the root and weight lattices of $G$. Define the following group
\[Q(G) = (\mathrm{Sym}^2(\Lambda)^W)^{\Gamma}\]
where $W$ is the Weyl group. It can also be described as the group of $\Gamma$-equivariant loops in $G$ \cite[\S 31]{KMRT}. For a  simply connected group, $\Lambda = \Lambda_w$, and $Q(G)$ is generated by a single element denoted $q$ \cite[Cor.~31.27]{KMRT}. Thus, for any other $\Lambda$, since $(\mathrm{Sym}^2(\Lambda)^W) = \mathrm{Sym}^2(\Lambda) \cap (\mathrm{Sym}^2(\Lambda_w)^W)$, there is a unique positive integer $\ell$ such that $Q(G) = \ell q\mathbb{Z}$. Furthermore, $\ell$ is the smallest integer such that the quadratic form $\ell q$ takes only integer values on the lattice $\Lambda$.

Let $n_G$ be the gcd of all Dynkin indices of all representations of $G$. The values of this number for absolutely simple simply connected groups can be found in \cite[Appendix B]{GMS}. Since $\ell \mid n_G$, 
\[\Dec(G) = n_Gq\mathbb{Z}\]
defines a subgroup of $Q(G)$, cf. \cite[Ex.~3.5]{ssinv2} . 

\begin{lem}
For $G=\SL_n/\mu_m$ and $G = \Hsp_{4n}$,
\[\Invi \simeq Q(G)/\Dec{G}.\]
\end{lem}
\begin{proof}
This follows immediately from Merkurjev's exact sequence (\ref{thm:mes}) and the remarks following \cite[Thm.~3.9]{ssinv2}, since, in the case of split groups, the map $\theta^*_G$ is trivial.
\end{proof}

Therefore, in order to calculate the indecomposable invariants, it suffices to compute this quotient. In the following  sections, we compute the groups $Q(G)$ and $\Dec(G)$ for the split groups $G=\SL_n/\mu_{p^r}$  and $G = \Hsp_{4n}$. Throughout, $\ell q$ where $\ell\in \Z^+$ will denote the generator of $Q(G)$.

\section{$\SL_n/\mu_m$}

Let $p^s$ be the largest power of $p$ dividing $n$, $r$ a positive integer with $r\leq s$. 

\begin{thm}\label{thm:InvSL}
\[\mathrm{Inv}^3(\SL_n/\mu_{p^r},\mathbb{Q}/\mathbb{Z}(2))_{\mathrm{dec}}\cong F^{\times}/F^{\times p^r}.\]

Moreover, for $p$ odd,

\[\mathrm{Inv}^3(\SL_n/\mu_{p^r},\mathbb{Q}/\mathbb{Z}(2))_{\mathrm{ind}}\cong \begin{cases} (\Z/p^{r}\Z)q&\quad \text{if $s\geq  2r$}\\
(p^{2r-s}\Z/p^{r}\Z)q& \quad\text{if $r < s < 2r$}\\
0 & \quad\text{if $s=r$}\\
\end{cases}\]

and for $p=2$,

\[\mathrm{Inv}^3(\SL_n/\mu_{2^r},\mathbb{Q}/\mathbb{Z}(2))_{\mathrm{ind}}\cong \begin{cases} (\Z/2^{r}\Z)q&\quad \text{if $s\geq  2r+1$}\\
(2^{2r+1-s}\Z/2^{r}\Z)q& \quad\text{if $r +1 < s < 2r+1$}\\
0& \quad\text{if $s =r, r+1$}\\
\end{cases}\]

\end{thm}

\begin{proof}
This will follow from formulas (\ref{eq:QSL}) and (\ref{eq:nSL}) below.
\end{proof}

\subsection{$Q(G)$ for $\SL_n/\mu_m$}

For $\SL_n/\mu_m$, $\Lambda$ is generated by the coroots along with the element $\tau:=\frac{1}{m}(\alpha_1+2\alpha_2+\cdots+(n-1)\alpha_{n-1})$, see \cite{G:vanish}. In this case, since all the coroots have the same length, the quadratic form is just given by taking the Gram matrix to be the Cartan matrix and we get

\[q=\sum_{i=1}^n w_i^2-\sum_{i=1}^{n-1}w_iw_{i+1}.\]

We then have

\begin{align*}
q(\tau):&=\frac{1}{m^2}\left(\sum_{i=1}^{n-1}i^2-\sum_{i=1}^{n-2}i(i+1)\right)\\
&=\frac{1}{m^2}\left((n-1)^2+\frac{(n-2)(n-1)}{2}\right)\\
&=\frac{n(n-1)}{2m^2}.
\end{align*}
By definition, $m\mid n$, so the fact that $\gcd(n-1,n)=1$ implies that we have 
\begin{equation}\label{eq:QSL}
\ell=\begin{cases}
2m^2/\gcd(2m^2,n)& \quad \text{if $n$ is even.}\\
m^2/\gcd(m^2,n)& \quad \text{if $n$ is odd.}
\end{cases}
\end{equation}

Note that if $m=n$ then this agrees with Merkurjev's result for adjoint groups, and also says, of course, that if $m=1$, i.e. $G$ is simply connected, then $\ell=1$.

\subsection{$\Dec(G)$ for $\SL_n/\mu_{p^r}$}

For a group of type $A_{n-1}$, we take the set of simple roots $\lbrace \overline{e}_{1}-\overline{e}_{2},...,\overline{e}_{n-1}-\overline{e}_{n}\rbrace$, where the $\overline{e}_i$ are the images of the standard basis vectors $e_i$ for $\mathbb{R}^n$. A dominant character $\chi \in \Lambda$ corresponds to a sum $\sum c_i \overline{e}_i$ with $c_1 \geq c_2 \geq ... \geq c_n$. Suppose the $c_i$ have distinct values $a_1 > ... >a_k$ with multiplicities $r_1,...,r_k$. In this case, by \cite[pg. 136]{GMS}, $n_G$ can be computed by taking the gcd over all integers
\[N(\chi) = \frac{(n-2)!}{r_1!r_2!...r_k!}\lbrack n \sum_i r_ia_i^2 - (\sum_ir_ia_i)^2 \rbrack.\]

We already know some convenient bounds on $n_G$. First, by \cite[Part 2, Lem.~11.4]{GMS}, $m \mid n_G$. Moreover, $n_G \mid \gcd(m^2,2n)$. To see that it divides the first, choose the dominant character with $c_1=m$, $c_i=0$ for $i>0$. It divides the second by \cite[Ex.~1.3]{G:vanish} and the observation that the Coxeter number in this case is $n$. 

We consider only the case where $m=p^r$ is a power of a prime and $n$ is arbitrary. In this case, $n_G$ must be a power of $p$ because it divides $m^2$. Let $n=k\cdot p^s$ where $k$ is coprime to $p$. If $s=r$, then we can conclude by the above bounds. Namely, if $p\neq 2$ then we have $n_G=p^r$ since we know it is a power of $p$, and it has to be equal to $p^r$ because it divides $2n$. If $p=2$, then as $Q(\SL_{n}/\mu_{2^r})$ has generator $2^{r+1}q$ in this case, we have that $2^{r+1} \mid n_G$. Further, $n_G$ must be $2^{r+1}$ because it is a power of $2$ dividing $2n$. 

Thus we reduce to assuming that $s \not= r$. Consider the $m$-th exterior power of the tautological representation of $\SL_n$. From \cite{Bou:alg9}, it has highest weight $\lambda=e_1+\cdots+e_m$, and in the above notation, $a_1=1, a_2=0$ and $r_1=m, r_2=n-m$, so that
\begin{align*}
N(\lambda)&=\frac{(n-2)!}{m!(n-m)!}(n\cdot m-m^2)\\
&=\binom{n-2}{m-1}\\
&=\binom{n}{m}\frac{m(n-m)}{n(n-1)}
\end{align*}
as can be found in Dynkin's Tables \cite[Table 5]{Dynk:ssub}. Appealing to Kummer's Theorem on the power of a prime dividing the binomial coefficients, we have
\[\ord_p\binom{n}{p^r} = s-r.\]

Using this in the above formula, if $s \neq r$,
\begin{align*}
\ord_p\binom{n}{p^r}\frac{p^r(kp^s-p^r)}{kp^s(kp^s-1)}&=\ord_p\binom{n}{p^r}+2r-s\\
&=s-r+2r-s\\
&=r
\end{align*}
This computation implies that $n_G=p^r$, since we already knew that $m=p^r \mid n_G$. 

In summary,
\begin{equation}\label{eq:nSL}
n_{\SL_n/\mu_{p^r}} = \begin{cases} p^r & \text{if $p \not = 2$ or $p=2$ and $s\not=r$}.\\
2^{r+1}& \text{if $p=2$ and $s=r$}.
\end{cases}
\end{equation}

\subsection{A Fibration}
It would be nice to have an explicit description of the indecomposable invariants in this case. However, even the torsors for these groups are difficult to describe. Lacking a reference, we include here a fibration which in some small way explains these objects.
\begin{prop}
There is a surjection of pointed sets
\[H^1(F,\SL_n/\mu_{m}) \twoheadrightarrow \lbrace \text{central simple algebras/$F$ of degree $n$ with exponent $\mid m$} \rbrace/\sim\]
where $A \sim B$ if $A$ and $B$ are isomorphic as $F$-algebras. Moreover, the fiber over the algebra $A$ is isomorphic to $F^{\times}/\mathrm{Nrd}(A^{\times})F^{\times n/p^r}$.
\end{prop}
\begin{proof}
Consider the commutative diagram:
\[\begin{CD}
1 @>>> \mu_{m} @>>> \SL_n @>>> \SL_n/\mu_{m} @>>> 1\\
@. @VVV @VVV @VVV @.\\
1 @>>> \mathbb{G}_m @>>> \mathrm{GL}_n @>>> \mathrm{PGL}_n @>>> 1\\
\end{CD}\]
Passing to cohomology, we have that
\[\begin{CD}
H^1(F,\SL_n/\mu_{m}) @>>> H^2(F,\mu_{m})\\
@V\phi VV @VVV\\
H^1(F,\mathrm{PGL}_n) @>>> H^2(F,\mathbb{G}_m)
\end{CD}\]
commutes. Since
\[\mathrm{Im}(\phi) = \mathrm{ker}\lbrace H^1(F,\mathrm{PGL}_n) \longrightarrow H^2(F,\mu_{n/m}) \rbrace,\]
it follows that $H^1(F,\SL_n/\mu_{m})$ maps onto the isomorphism classes of central simple algebras of degree $n$ with exponent dividing $m$. The fiber of the top map of the diagram over an algebra $A$ is in the image of $H^1(F,\mathrm{SL}_1(A))$, and this set can be easily computed from the sequence
\[\begin{CD}
1 @>>> \mathrm{SL}_1(A)(F) @>>> \mathrm{GL}_1(A)(F) @>>> \mathbb{G}_m(F) @>>> H^1(F,\mathrm{SL}_1(A)) @>>> 1.
\end{CD}\]
Finally, the map $H^1(F,\mu_{m}) \rightarrow H^1(F,\mathrm{SL}_1(A))$ sends a field element $a \in F^{\times}$ to $a^{n/m} \in F^{\times}/\mathrm{Nrd}(A^{\times})$. This yields the fibration described in the proposition.
\end{proof}

\subsection{Examples}
If $p=2$, then any central simple algebra with exponent dividing $2$ has an involution of the first kind. This structure has been used to define degree 3 cohomological invariants in some small cases \cite{GPT}, but for odd primes and higher powers, these invariants do not exist in the literature. However, for some values we have a complete description.

\begin{thm}\label{thm:SL2mu2}
If $4 \mid n$,
\[\mathrm{Inv}^3(\SL_{2n}/\mu_2,\mathbb{Q}/\mathbb{Z}(2))_{\mathrm{norm}} \simeq F^{\times}/F^{\times 2} \oplus \Z/2\Z.\]
Moreover, these invariants can be described explicitly by restricting the degree $3$ invariants of $\mathrm{PSp}_{2n}$.
\end{thm}
\begin{proof}
Using the natural inclusion
\[\mathrm{Sp}_{2n} \subset \SL_{2n}\]
and modding out by the center of $\mathrm{Sp}_{2n}$ gives an inclusion
\[\mathrm{PSp}_{2n} \subset \SL_{2n}/\mu_2.\]

The degree $3$ invariants of $\mathrm{PSp}_{2n}$ were completely described in \cite{GPT} when $4 \mid n$. We use this description and the above inclusion to do the same for $\SL_{2n}/\mu_2$. 
The group $\mathrm{SL}_{2n}$ acts on the $2n$-dimensional vector space $F^{\oplus 2n}$. Let $V = \wedge^2F^{\oplus 2n}$. $\mathrm{SL}_{2n}$ acts canonically on $V$, and since it is the second exterior power, so does $\SL_{2n}/\mu_{2}$. Over an algebraic closure, there is an open $\SL_{2n}/\mu_{2}$-orbit in $\mathbb{P}(V)$ \cite[Summary Table]{PoV}, so by\cite[\S 9.3]{G:Lens}, there is a surjection
\[\begin{CD}
H^1(F,N) @>>> H^1(F, \SL_{2n}/\mu_2)
\end{CD}\]
where $N$ is the stabilizer of a generic point in the open orbit. By inspection, or again referring to \cite[Summary Table]{PoV}, $N = \mathrm{Sp}_{2n}/\mu_2 = \mathrm{PSp}_{2n}$. Thus we have an inclusion
\[\mathrm{Inv}^3(\SL_{2n}/\mu_2,\mathbb{Q}/\mathbb{Z}(2))_{\mathrm{norm}} \hookrightarrow \mathrm{Inv}^3(\PSp_{2n},\mathbb{Q}/\mathbb{Z}(2))_{\mathrm{norm}}.\]
Now, from the commutative diagram
\[\begin{CD}
H^1(F, \PSp_{2n}) @>>> H^2(F, \mu_2) \\
@VVV @|\\
H^1(F, \SL_{2n}/\mu_2) @>>> H^2(F, \mu_2),\\
\end{CD}\]
and the description of the decomposable invariants in \cite[Thm.~4.6]{ssinv2}, it follows that 
\[\mathrm{Inv}^3(\SL_{2n}/\mu_2,\mathbb{Q}/\mathbb{Z}(2))_{\mathrm{dec}} \simeq \mathrm{Inv}^3(\PSp_{2n},\mathbb{Q}/\mathbb{Z}(2))_{\mathrm{dec}}.\]
By the commutativity of diagram (\ref{eqn:decinv}), it follows that \[\mathrm{Inv}^3(\SL_{2n}/\mu_2,\mathbb{Q}/\mathbb{Z}(2))_{\mathrm{dec}} \hookrightarrow \mathrm{Inv}^3(\SL_{2n}/\mu_2,\mathbb{Q}/\mathbb{Z}(2))_{\mathrm{norm}}\] also splits, and these groups have the same degree $3$ normalized invariants. The statement follows from the description of the invariants in Theorem \ref{thm:InvSL} and the explicit construction of \cite{GPT}.
\end{proof}

\subsection*{An Explicit Invariant for $\mathrm{SL}_{8}/\mu_2$}

In \cite[\S 5]{GPT}, the authors go further to describe the invariants for $\mathrm{PSp}_8$ via the Rost invariant for a simply connected group of type $E_6$. A similar calculation can be done for $\SL_8/\mu_2$ and a simply connected group of type $E_7$. First, consider the diagram
\[\begin{array}[c]{ccc}
\mathrm{PSp}_8 &\subset& \mathrm{E}_6\\
\cap & &\cap\\
\mathrm{SL}_8/\mu_2 &\subset& \mathrm{E}_7
\end{array}\]
The inclusion $\mathrm{SL}_8/\mu_2 \hookrightarrow \mathrm{E}_7$ can be obtained via the Borel-de Siebenthal theory of maximal rank subgroups by deleting the vertex labeled $2$ which does not disconnect the Dynkin diagram, cf. \cite{Lehalleur}. We have the corresponding map of simply connected groups
\[\mathrm{SL}_8 \rightarrow \mathrm{SL}_8/\mu_2 \hookrightarrow \mathrm{E}_7.\]
By inspection, the Rost multiplier is $1$ in this case, so the non-trivial degree $3$ indecomposable invariant for $\mathrm{SL}_8/\mu_2$ is a restriction of the Rost invariant of $E_7$. We also remark here that the invariant $\Delta(A,\sigma)$ described in \cite{GPT} is a generator for the indecomposable invariants of $\mathrm{PSp}_8$ and thus of $\mathrm{SL}_8/\mu_2$ by Theorem \ref{thm:SL2mu2}. The decomposable invariants $x \in F^{\times}/F^{\times 2}$ are defined on an element $y \in H^1(F,\SL_8/\mu_2)$ as $y \mapsto (x) \cup \partial(y)$ where $\partial: H^1(F,\SL_8/\mu_2) \rightarrow  H^2(F,\mu_2)$. From the definition of $\Delta(A,\sigma)$ it is difficult to say much more about it. However, in a recent preprint, Demba Barry \cite{Barry} has shown that the image of this element in $\frac{H^3(F,\mu_2)}{[A] \cdot F^{\times}}$ is non-zero for an indecomposable algebra of degree $8$ and exponent $2$. That is, not only is $\Delta$ not generated by decomposable invariants, it cannot even be written in a similar manner using cup products.

\subsection*{An Explicit Invariant for $\mathrm{SL}_{9}/\mu_3$}

Similarly to the previous example, the indecomposable invariants of $\mathrm{SL}_{9}/\mu_3$ can be described via a simply connected group of type $E_8$. Again, the Borel-de Siebenthal theory shows that by deleting the vertex labeled $3$ which does not disconnect the Dynkin diagram, there is an inclusion of groups
\[\mathrm{SL}_9/\mu_3 \hookrightarrow \mathrm{E}_8.\]
As above, the Rost multiplier for the induced map on simply connected groups is $1$. Therefore, a generator for the degree $3$ indecomposable invariants of $\mathrm{SL}_9/\mu_3$ is the restriction of the Rost invariant of $E_8$. Moreover, the $3$-torsion part of $\mathrm{Inv}_{\mathrm{norm}}(E_8,\mathbb{Q}/\Z(2))$ is isomorphic to $\Z/3\Z$ \cite[Part 2, Thm.~16.8]{GMS}, so the restriction gives a splitting:
\[\mathrm{Inv}^3(\mathrm{SL}_9/\mu_3,\mathbb{Q}/\Z(2))_{\mathrm{norm}} \simeq F^{\times}/F^{\times 3} \oplus \Z/3\Z.\]

\section{$\Hsp$}

\begin{thm}
We have 
\[\mathrm{Inv}^3(\Hsp_{4n},\mathbb{Q}/\mathbb{Z}(2))_{\mathrm{ind}}\cong \begin{cases} 0&\quad \text{if $n>1$ is odd or $n=2$}\\
2\Z/4\Z& \quad\text{if $n\equiv 2\mod 4$ and $n\neq 2$}\\ 
\Z/4\Z& \quad\text{if $n\equiv 0\mod 4$}
\end{cases}\] 
\end{thm}

\begin{proof}
This follows from equations (\ref{eqn:QHSp}) and (\ref{eqn:nHSp}) below.
\end{proof}

\begin{cor}\label{cor:Hsp16}
We have
\[\mathrm{Inv}^3(\Hsp_{16},\mathbb{Q}/\mathbb{Z}(2))_{\mathrm{norm}}\cong F/F^{\times2}\oplus \Z/4\Z.\]
\end{cor}
\begin{proof}
The invariant $e_3$ for $\Hsp_{16}$ constructed in \cite{G:deg16} via the inclusion $\Hsp_{16}\to E_8$, gives the required splitting. 
\end{proof}

\subsection{$Q(G)$ for $\Hsp$}

As in the previous case, for $G$ of type $D_{2n}$ all the roots have the same lengths, and $\Lambda$ is generated by the coroots and the additional element $\tau=\frac{1}{2}\sum_{i \text{ odd}} \alpha_i$. The quadratic form $q$ is given by
\[q=\sum_{i=1}^{2n} w_i^2-\sum_{i=1}^{2n-2}w_iw_{i+1}-w_{2n-2}w_{2n}.\]
Computing as before, $q(\tau)=\frac{n}{4}$, and it follows that
\begin{equation}\label{eqn:QHSp}
\ell=\begin{cases}
1&\quad\text{if $n\equiv 0\mod 4$.}\\
2&\quad\text{if $n\equiv 2\mod 4$.}\\
4&\quad\text{if $n$ is odd.}
\end{cases}
\end{equation}

\subsection{$\Dec(G)$ for $\Hsp_{4n}$}

Let $\tilde{G}\to G\to \overline{G}$ be the standard central isogenies where $\tilde{G}$ is the simply connected cover of $G$ and $\overline{G}=G/C(G)$ is adjoint. We have the following relationship for the Dynkin indices: 
\[n_{\tilde{G}}\mid n_G\mid n_{\overline{G}}.\]
This follows from the definition: $n_{\tilde{G}}$ is the $\gcd$ over all representations given by highest weights in $\Lambda_w$, whereas $n_G$ is the $\gcd$ taken over all representations whose weights vanish on the kernel of the map $\tilde{G}\to G$. Similarly $n_{\overline{G}}$ is the gcd of representations whose highest weight vanishes on all of the kernel of $\tilde{G}\to \overline{G}$. Note that for $\tilde{G}$, the Dynkin index is equal to the order of the group of degree 3 invariants see \cite[Part 2, Thm.~10.7]{GMS}. Applying this to the case where $G=\Hsp_{4n}$ we get that
\[2\mid n_{\Hsp_{4n}}\mid 4,\]
since we have $n_{\overline{G}}=n_{\PGO_{4n}}=4$ from \cite{ssinv2} and $n_{\tilde{G}}=n_{\Spin_{4n}}=2$ or $4$ from \cite[Appendix B]{GMS}. 

Now let $\chi$ be a fundamental weight of $\Spin_{4n}$, $C$ the center of $\Spin_{4n}$; then $C^*_{\mathrm{sep}}$ consists of four elements, $0, \lambda, \lambda^+$ and $\lambda^-$ (see \cite[p.~146]{GMS}). Put $n_0,n^+$ for the $\gcd(N(\chi))$ where the $\gcd$ is taken over all characters that restrict to 0, and $\lambda^+$ respectively, then one has that $n_{\Hsp_{4n}}=\gcd(n_0,n^+\cdot \mathrm{ind}(C^+))$. Furthermore from \cite[Part 2, Lem.~15.3]{GMS}, we know that $n^+=2^{2n-3}$ and $n_0$ is divisible by 4, and this implies that 
\begin{equation}
\label{eqn:nHSp}
\begin{aligned}
&n_{\Hsp_8}=2.\\ 
&n_{\Hsp_{4n}}=4 \text{ for $n>2$.}
\end{aligned}
\end{equation}

\section{Restriction of Invariants to Subgroups}

\subsection{Restrictions in terms of $Q(G)/\Dec(G)$}

Consider the following diagram of groups 
\begin{equation}\label{eqn:rest}\xymatrix{
\mu_2\ar[d]\ar@{=}[r]&\mu_2\ar[d]\ar@{=}[r]&\mu_2\ar[d]\\
\Sp_{2n}\ar[d]\ar@{^{(}->}[r]&\SL_{2n}\ar[d]\ar@{^{(}->}[r]&\Spin_{4n}\ar[d]\\
\PSp_{2n}\ar@{^{(}->}[r]&\SL_{2n}/\mu_2\ar@{^{(}->}[r]&\Hsp_{4n}}
\end{equation}
where the left vertical sequence is just the standard isogeny. The inclusions 
\[\Sp_{2n}\subset\SL_{2n}\subset\Spin_{4n}\]
can be easily described. The first was treated above. The second comes from deleting an appropriate end vertex of the Dynkin diagram of type $D_{2n}$ to obtain a diagram of type $A_{2n-1}$ in such a way that $\mu_2$ will still sit inside all groups, and we will be able to obtain the inclusion $\SL_{2n}/\mu_2\subset \Hsp_{4n}$ as claimed. Now notice that by right side of diagram (\ref{eqn:decinv}) and the previous results we get a diagram
\[\xymatrix{
\mathrm{Inv}^3(\Hsp_{4n},\mathbb{Q}/\mathbb{Z}(2))_{\mathrm{ind}}\ar[d]^{\simeq}\ar[r]&\mathrm{Inv}^3(\SL_{2n}/\mu_2,\mathbb{Q}/\mathbb{Z}(2))_{\mathrm{ind}}\ar[d]^{\simeq}\ar[r]&\mathrm{Inv}^3(\PSp_{2n},\mathbb{Q}/\mathbb{Z}(2))_{\mathrm{ind}}\ar[d]^{\simeq}\\
Q(\Hsp_{4n})/\Dec(\Hsp_{4n})\ar[r]&Q(\SL_{2n}/\mu_2)/\Dec(\SL_{2n}/\mu_2)\ar[r]&Q(\PSp_{2n})/\Dec(\PSp_{2n})
}\]
The top row of this diagram gives the restriction of the generator of the group of indecomposable invariants of $\Hsp_{4n}$ to $\SL_{2n}/\mu_2$ and the restriction of the generator of the group of indecomposable invariants of $\SL_{2n}/\mu_2$ to $\PSp_{2n}$ which was described above. We now compute the other restriction by using the bottom row of the diagram. 

Notice that the only interesting case is when $n\equiv 0\mod 4$ since if $n$ is odd then all of the groups of indecomposable invariants are trivial, and there is nothing to say. Similarly if $n\equiv 2\mod 4$ then $\mathrm{Inv}^3(\Hsp_{4n},\mathbb{Q}/\mathbb{Z}(2))_{\mathrm{ind}}\cong\Z/2\Z$ as shown above. However, the other two groups are both trivial, so the generator restricts trivially in these cases. To finish we note that, since the coroots of $\SL_{2n}$ are also coroots of $\Spin_{4n}$ and they all have the same length, the Rost multiplier of the inclusion $\SL_{2n}\hookrightarrow \Spin_{4n}$ is 1. This means that the map
\[\Z/4\Z \simeq Q(\Hsp_{4n})/\Dec(\Hsp_{4n})\rightarrow Q(\SL_{2n}/\mu_2)/\Dec(\SL_{2n}/\mu_2) \simeq \Z/2\Z\]
maps $1\in \Z/4\Z$ to $1\in \Z/2\Z$. Therefore, a generator of the group of indecomposable invariants of $\Hsp_{4n}$ restricts to a generator of the indecomposable invariants of $\SL_{2n}/\mu_2$.

We can go further and note that since we have 
\[\mathrm{Inv}^3(\Hsp_{4n},\mathbb{Q}/\mathbb{Z}(2))_{\mathrm{dec}}\simeq \mathrm{Inv}^3(\SL_{2n}/\mu_2,\mathbb{Q}/\mathbb{Z}(2))_{\mathrm{dec}}\simeq \mathrm{Inv}^3(\PSp_{2n},\mathbb{Q}/\mathbb{Z}(2))_{\mathrm{dec}}\]
by applying the five lemma to diagram (\ref{eqn:decinv}) we get that the map
\begin{align*}
&\mathrm{Inv}^3(\Hsp_{4n},\mathbb{Q}/\mathbb{Z}(2))_{\mathrm{norm}}\to \mathrm{Inv}^3(\SL_{2n}/\mu_2,\mathbb{Q}/\mathbb{Z}(2))_{\mathrm{norm}}
\end{align*}
is onto.
\subsection{Explicit description of restrictions}

Consider diagram (\ref{eqn:rest}), by looking at the long cohomology exact sequence we get a diagram:

\[\xymatrix{
H^1(F,\Sp_{2n})\ar[r]\ar[d]&H^1(F,\PSp_{2n})\ar[d]\ar[r]&H^2(F,\mu_2)\ar@{=}[d]\\
H^1(F,\SL_{2n})\ar[r]\ar[d]&H^1(F,\SL_{2n}/\mu_2)\ar[d]\ar[r]&H^2(F,\mu_2)\ar@{=}[d]\\
H^1(F,\Spin_{4n})\ar[r]&H^1(F,\Hsp_{4n})\ar[r]&H^2(F,\mu_2)
}\]

Now let $x\in H^1(F,\Hsp_{4n})$ be a class mapping to a class $[A]\in H^2(F,\mu_2)$ with index dividing $2n$; $z$ and element of $H^1(F,\PSp_{2n})$ which also maps to $[A]$, and let $x_0$ be the image of $z$ in $H^1(F,\Hsp_{4n})$. Twisting the bottom row of this diagram by $x_0$ and putting $\Spin_{x_0}$ for the twist of $\Spin_{4n}$ we find a cocycle $y\in H^1(F,\Spin_{x_0})$ which maps to $x\in H^1(F,\Hsp_{4n})$.

Notice that the map on indecomposable invariants induced by the quotient $\Spin_{x_0}\to \Hsp_{x_0}$ is onto because the group $Q(\Hsp_{x_0})$ contains the generator $q$ of $Q(\Spin_{x_0})$. That is, there is a generator $e_3'$ of $\mathrm{Inv}^3(\Hsp_{x_0},\mathbb{Q}/\mathbb{Z}(2))_{\mathrm{ind}}$ which maps to the Rost invariant $e_3^{\Spin}$ of $\Spin_{x_0}$. Now let $e_3$ be the image of $e_3'$ in $\mathrm{Inv}^3(\Hsp_{4n},\mathbb{Q}/\mathbb{Z}(2))_{\mathrm{ind}}$ under the isomorphism described in \cite[p.~14]{ssinv2}. We get an equation
\[e_3(x)=e_3^{\Spin}(y)+e_3(x_0)\in H^3(F,\Z/4\Z)/P,\]
where $P$ here is the subgroup defined in section 3 above, namely the subgroup generated by cup products of Tits algebras with elements of the field. By the results of the last section we obtain:
\begin{prop}
Let $\Delta$ denote the invariant of $\mathrm{Inv}^3(\PSp_{2n},\mathbb{Q}/\mathbb{Z}(2))_{\mathrm{norm}}$ (this invariant was first constructed in \cite{GPT}), by the results of the last section we obtain the equation
\begin{equation}\label{eqn:res}
e_3(x)=\Delta(z)+e_3^{\Spin}(y)\, \in H^3(F,\Z/4\Z)/P.
\end{equation}
\end{prop}

\begin{rmk}
The case $n=4$ of the previous proposition was first proven in \cite[Cor.~10.2]{G:deg16}. For this case, Corollary \ref{cor:Hsp16} allows us to strengthen the statement of the proposition to an equation
\begin{equation}
e_3(x)=\Delta(z)+e_3^{\Spin}(y)\, \in H^3(F,\Z/4\Z).
\end{equation}
\end{rmk}

\section{Algebras with orthogonal involution in $I^3$}

In \cite{G:deg16}, Garibaldi uses a construction of a degree 3 invariant of $\Hsp_{16}$ to define an invariant for central simple algebras $(A,\sigma)$ of degree 16 with an orthogonal involution and to deduce some nice properties. We now wish to show how the results of this paper allow us to recover and extend some of those results to algebras of degree any multiple of 16.

Let $(A,\sigma)$ be a central simple algebra with orthogonal involution over a field $F$ of characteristic not 2. Over the function field $F_A$ of the Severi-Brauer variety of $A$, the involution $\sigma$ is adjoint to a quadratic form $q_\sigma$ determined up to similarity. We say $\sigma$ \emph{is in $I^n$}, if $q_\sigma$ is in $I^n$, where $I$ denotes the fundamental ideal in the Witt ring of $F_A$.

To relate these algebras with the results of this paper we recall that by \cite[Lem.~4.1]{G:deg16} the pairs $(A,\sigma)$ which lie in $I^3$ are exactly those that are in the image of the map $H^1(k,\Hsp_{4n})\to H^1(k,\PGO_{4n})$.

Now let $e_3$ be a generator of $\mathrm{Inv}^3(\Hsp_{4n},\mathbb{Q}/\mathbb{Z}(2))_{\mathrm{ind}}$, and for a given pair $(A,\sigma)$ fix an element $\eta\in H^1(F,\Hsp_{4n})$ which maps to $(A,\sigma)$. Define
\[e_3(A,\sigma):=e_3(\eta)\in H^3(F,\Q/\Z(2))/[A]\cdot H^1(F,\mu_2).\]
We note that the quotient on the right hand side implies that the value of $e_3(A,\sigma)$ does not depend on the choice of $\eta$ since we have the sequence:
\[H^1(F,\mu_2)\to H^1(F,\Hsp_{4n})\to H^1(F,\PGO_{4n})\]

Put $E(A):=\ker(H^3(F,\Z/4\Z)\to H^3(F_A,\Z/4\Z))$, and let $P=[A]\cdot H^1(F,\mu_2)$ as before. Clearly $P\subset E(A)$ since $A$ splits over $F_A$. We now obtain the following generalization of \cite[Thm.~2.6, Cor.~2.8]{G:deg16}

\begin{thm}
Let $(A,\sigma)$ be a central simple algebra with orthogonal involution of degree divisible by 16. Then there exists an invariant $e_3(A,\sigma)\in H^3(F,\Q/\Z(2))/E(A)$ such that if $K/F$ splits $A$, then $\res_{K/F} e_3(A,\sigma)$ is the Arason invariant $e_3(q_{\sigma\otimes K})$. Furthermore $(A,\sigma)$ is in $I^4$ if and only if $e_3(A,\sigma)$ is zero.
\end{thm}

\begin{proof}
The construction of the invariant was described above. The proofs of the other statements can be taken basically verbatim from \cite{G:deg16}. We sketch them here for the reader's convenience. Suppose the algebra $A$ splits over $K$, then by the results of section 6 above, there exists a class $x\in H^1(K,\Spin_{4n})$ which maps to $(A,\sigma)$ in $H^1(K,\PGO_{4n})$, and the value of the restriction $\res_{K/F}e_3(A,\sigma)$ is the value of the Rost invariant of $x$, i.e. the Arason invariant of $q_\sigma$. That $(A,\sigma)$ is in $I^4$ if and only if $e_3(A,\sigma)$ is zero then follows from the corresponding statement for $q_\sigma$ and the Arason invariant.
\end{proof}

\begin{rmk}
It is shown in \cite[Thm.~3.9]{BPQ} that a degree 3 invariant restricting to the Arason invariant does not exist for degree 8 algebras with orthogonal involution in $I^3$. These results are extended in \cite[Ex.~2.7]{G:deg16} to rule out the existence of an invariant in all degrees not covered by the theorem above.
\end{rmk}

{\footnotesize
\textit{Acknowledgments.} This article was written while the first author was a visitor at the Fields Institute in Toronto. He wishes to thank the staff at the Institute and the organizers of the Thematic Program on Torsors, Nonassociative Algebras and Cohomological Invariants for their hospitality. The second author would like to thank Emory University for its continuing support.}

\bibliography{skip_master}

\begin{thebibliography}{10}

\bibitem{Barry}
D.~Barry.
\newblock Decomposable and indecomposable algebras of degree 8 and exponent 2.
\newblock \url{http://www.math.uni-bielefeld.de/LAG/man/497.pdf}.

\bibitem{BPQ}
E.~Bayer-Fluckiger, R.~Parimala, and A.~Qu{\'e}guiner-Mathieu.
\newblock Pfister involutions.
\newblock {\em Proc. Indian Acad. Sci. Math. Sci.}, 113(4):365--377, 2003.

\bibitem{Bou:alg9}
N.~Bourbaki.
\newblock {\em Alg{\`e}bre {IX}}.
\newblock Hermann, Paris, 1959.

\bibitem{Dynk:ssub}
E.B. Dynkin.
\newblock Semisimple subalgebras of semisimple {L}ie algebras.
\newblock {\em Amer. Math. Soc. Transl. (2)}, 6:111--244, 1957.
\newblock [Russian original: Mat.\ Sbornik N.S.\ \textbf{30(72)} (1952),
  349--462].

\bibitem{G:Lens}
S.~Garibaldi.
\newblock {\em Cohomological invariants: exceptional groups and spin groups}.
\newblock Number 937 in Memoirs Amer. Math. Soc. Amer. Math. Soc., Providence,
  RI, 2009.
\newblock with an appendix by Detlev W. Hoffmann.

\bibitem{G:deg16}
S.~Garibaldi.
\newblock Orthogonal involutions on algebras of degree $16$ and the {K}illing
  form of ${E}_8$.
\newblock In R.~Baeza, W.K. Chan, D.W. Hoffmann, and R.~Schulze-Pillot,
  editors, {\em Quadratic forms--algebra, arithmetic, and geometry}, volume 493
  of {\em Contemp. Math.}, pages 131--162, 2009.
\newblock with an appendix by Kirill Zainoulline.

\bibitem{G:vanish}
S.~Garibaldi.
\newblock Vanishing of trace forms in low characteristic.
\newblock {\em Algebra \& Number Theory}, 3(5):543--566, 2009.
\newblock with an appendix by A. Premet.

\bibitem{GMS}
S.~Garibaldi, A.~Merkurjev, and J-P. Serre.
\newblock {\em Cohomological invariants in {G}alois cohomology}, volume~28 of
  {\em University Lecture Series}.
\newblock Amer.\ Math.\ Soc., 2003.

\bibitem{GPT}
S.~Garibaldi, R.~Parimala, and J.-P. Tignol.
\newblock Discriminant of symplectic involutions.
\newblock {\em Pure Applied Math. Quarterly}, 5(1):349--374, 2009.

\bibitem{Ill}
L.~Illusie.
\newblock Complexe de de rham-witt et cohomologie cristalline.
\newblock {\em Ann. Sci. {\'E}cole Norm. Sup.}, 12(4):501--661, 1979.

\bibitem{KMRT}
M.-A. Knus, A.S. Merkurjev, M.~Rost, and J.-P. Tignol.
\newblock {\em The book of involutions}, volume~44 of {\em Colloquium
  Publications}.
\newblock Amer.\ Math.\ Soc., 1998.

\bibitem{Lehalleur}
S.P. Lehalleur.
\newblock Subgroups of maximal ranks of reductive groups.
\newblock \url{ http://www.math.ens.fr/~gille/actes/pepin.pdf}.

\bibitem{ssinv2}
A.~Merkurjev.
\newblock Degree three cohomological invariants of semisimple groups.
\newblock \url{http://www.math.uni-bielefeld.de/LAG/man/495.html}.

\bibitem{PoV}
V.L. Popov and E.B. Vinberg.
\newblock {\em Invariant theory}, volume~55 of {\em Encyclopedia of
  Mathematical Sciences}, pages 123--284.
\newblock Springer-Verlag, 1994.

\end{thebibliography}
	\bibliographystyle{plain}

\end{document}